\documentclass[12pt, reqno]{amsart}

\usepackage{amsmath, amsthm, amscd, amsfonts, amssymb, graphicx, color}
\usepackage[bookmarksnumbered, colorlinks, plainpages]{hyperref}

\newtheorem{theorem}{Theorem}[section]

\newtheorem{lem}[theorem]{Lemma}
\newtheorem{prop}[theorem]{Proposition}
  \newtheorem{defn}[theorem]{Definition}
  
\newtheorem{exam}[theorem]{Example}

\numberwithin{equation}{section}

\newcommand{\BC}{{\Bbb C}}

\newcommand{\BR}{{\Bbb  R}}

\newcommand{\BK}{{\Bbb K}}
\newcommand{\ext}{{\rm ext}}

\newcommand{\St}{{\rm St}}

\newcommand{\ch}{{\rm ch}}

\begin{document}

%  ===================  Tiltle =========================
\title[Isometries between completely regular vector-valued function
spaces]{Isometries between completely regular vector-valued function
spaces}

\author{Mojtaba  Mojahedi and  Fereshteh Sady$^1$}



\maketitle
\begin{center}

\address{{\em   Department of Pure
Mathematics, Faculty of \\ Mathematical Sciences,
 Tarbiat Modares University,\\ Tehran, 14115-134, Iran}}

\vspace*{.25cm}
  \email{mojtaba.mojahedi@modares.ac.ir},
  \email{sady@modares.ac.ir}
\end{center}
\footnote{$^1$ Corresponding author}

\maketitle

%  ===================== Abstract ==========================
\begin{abstract}
In this paper, first we study surjective isometries (not necessarily linear)  between completely regular subspaces $A$ and $B$ of $C_0(X,E)$ and $C_0(Y,F)$  where $X$ and $Y$ are locally compact Hausdorff spaces and $E$ and $F$ are normed spaces, not assumed to be  neither  strictly convex nor complete. We show that for a class of normed spaces  $F$ satisfying a new defined property related  to their  $T$-sets, such an isometry is a (generalized) weighted composition operator up to a translation. Then we apply the result to study surjective isometries between $A$ and $B$ whenever $A$ and $B$ are equipped with certain norms rather than the supremum norm. Our results unify and generalize some recent results in this context. 
 \end{abstract}

\maketitle

% ===================== Introduction =======================

\section{Introduction }
Considerable works have been done to characterize linear  isometries  between various Banach spaces of functions. The result for surjective isometries between Banach spaces of all continuous functions was initiated by Banach and Stone as the weighted composition operators. There are various generalizations of this theorem   based on different techniques. For instance, for the characterization of (surjective) isometries between subspaces of continuous scalar-valued functions endowed with the supremum norm or special complete norms, see \cite{Del,Font,Var1,Kaw,Roy,Rao-Roy}.
The first vector-valued version of  the Banach-Stone theorem was given by Jerison in \cite{Jer}. By Jerison's result, if $E$ is a strictly convex Banach space, then any surjective linear isometry on the Banach space $C(X,E)$ of all  $E$-valued continuous functions on a compact Hausdorff space $X$ is a generalized weighted composition operator.  There are similar results for the case where the dual space of $E$ is strictly convex \cite{Lau} and in general for the case that $E$ has trivial centralizer, see \cite[Cor. 7.4.11]{Flem2}. 

Surjective linear  isometries between subspaces  of  vector-valued Lipschitz functions with respect to particular complete norms  have been studied, for instance,  in   \cite{Ar, Bot1,Var2,Ran}. In \cite{Ar,Var2} the target spaces are assumed to be strictly convex. In \cite{Bot1} the values of the Lipschitz functions are taken in a quasi-sub-reflexive Banach space with trivial centralizer and this result has been improved in \cite{Ran} without the quasi-sub-reflexivity assumption.  We should note that in a strictly convex normed space $E$, any maximal convex subset of the unit sphere of $E$ is a singleton. However, in \cite{Jam} there are some results for surjective supremum  norm isometries between vector-valued  spaces of continuous functions with values in a Banach space $E$ whose unit sphere contains a maximal convex subset which is a singleton.  Characterizatin of  surjective  isometries on  spaces of vector-valued continuously differentiable functions with values in a finite-dimensional real Hilbert space can be found in \cite{Bot2}. Surjective isometries (not necessarily linear) between   spaces of vector-valued absolutely continuous functions  with values in a strictly convex normed space have been studied in \cite{Hos}. In the recent paper \cite{Hat} of Hatori, he studies  linear isometries between certain Banach algebras with values in $C(Y)$, where $Y$ is a  compact Hausdorff space. These Banach algebras include  the  $C(Y)$-valued Banach algebra of Lipschitz functions and $C(Y)$-valued Banach algebra of continuously differentiable functions. 
Recently, in \cite{MS}, isometries on certain subspaces of vector-valued continuous functions with respect to the supremum norm and the other (not necessarily complete) norms have been characterized. We should note that,  in \cite{MS}, neither the target space itself nor its dual space is  assumed to be strictly convex, but they satisfy a  mild condition related to the maximal convex subsets of the unit spheres.  

In this paper we deal with surjective (not necessarily linear) isometries between completely regular subspaces of vector-valued continuous functions endowed with either the  supremum  norm or other certain norms. Introducing a property,  called $(D_w)$, on the target spaces which is considerably weaker than the  strict convexity, we obtain some characterizations for such isometries as generalized weighted composition operators.

% ===================== Preliminaries ======================
\section{Preliminaries}

Throughout this paper $\BK$ stands for the scalar fields $\BR$ and
$\BC$. For a normed space $E$ over $\Bbb K$, we denote its closed unit ball by $E_1$
and its unit sphere by $S(E)$. We use the notations $E^*$ and ${\rm ext}(E^*_1)$
for the dual space of $E$ and the set of extreme points of $E^*_1$, respectively.

For a $\Bbb K$-normed space $E$, a  subset $S $ of  $E$ is said to be {\em  norm additive} if  for any finite collection $e_1,e_2,....,e_n$ of elements of $S$,
\[\|e_1+\cdots +e_n\|=\|e_1\|+\cdots \|e_n\|.\] A maximal norm additive subset of $E$ is called a {\em $T$-set} in $E$. If $e_1,e_2\in E$ such that  $\|e_1+e_2\|=\|e_1\|+\|e_2\|$, then for all $s,t\ge 0$ we have
$\|se_1+te_2\|=s \|e_1\|+t \|e_2\|$, see \cite[Lemma 4.1]{Jer}. Hence   for any $T$-set $S$ in $E$ and any $t\ge 0$ we have $tS \subseteq S$.

For each $e\in S(E)$ the star-like set ${\rm St}(e)$ is defined as
\[ \St(e)=\{ e'\in S(E): \|e+e'\|=2 \}.\] It is well-known that  $\St(e)$  is the union
of all maximal convex subsets of $S(E)$ containing $e$. Clearly, in the case that  $E$ is
strictly convex, we have  $\St(e)=\{e\}$ for all
$e\in S(E)$. We also note that if   $e\in S(E)$ such that  $\St(e)=\{e\}$, then $e$ is an extreme point of $E_1$, that is  $E$ is strictly convex if and only if $\St(e)=\{e\}$ for all $e\in S(E)$.
For each  $e\in S(E)$ and  $e'\in \St(e)$ we have $\|re+e'\|>r=\|re\|$ for $r>0$. Motivated by this, for each $u\in E$ we put
\[ \St_w(u)=\{e'\in S(E) : \|u+e'\|>\|u\|\}. \]
It should be noted that if $e'\in \St_w(u)$, then $\|u+re'\|> \|u\|$ for all $r\ge 1$, that is $e' \in  \St_w(\frac {u}{r})$  for all $r\ge 1$.

For a topological space $X$ and a normed space $E$ over $\Bbb K$, let $C(X,E)$ be the space of all continuous $E$-valued functions on $X$. For  an element $v\in E$,  the constant function $x \mapsto v$ in $C(X,E)$ will be denoted by $\hat{v}$.
In the  case that  $X$ is locally compact, $C_0(X,E)$ denotes the normed space of all continuous $E$-valued functions on $X$ vanishing at infinity, with the supremum norm $\|\cdot\|_\infty$.
By \cite[Theorem 2.3.5]{Flem1}, for  $\mathcal{Z}=C_0(X,E)$  we have
\[ {\rm ext}(\mathcal{Z}^*_1)=\{v^*\circ \delta_x: v^*\in {\rm ext}(E^*_1), x\in X\},\]
where for each $x\in X$, $\delta_x: C_0(X,E)\longrightarrow E$ is defined by $\delta_x(f)=f(x)$. Moreover,  if $A$ is a subspace of $C_0(X,E)$ then, by \cite[Corollary 2.3.6]{Flem1}, we have
\[ \ext(A^*_1)\subseteq \{v^*\circ \delta_x: v^*\in {\rm ext}(E^*_1), x\in X\}.\]
The {\em Choquet boundary} of $A$ which is denoted by  $\ch(A)$, consists of all points $x\in X$ such that   $v^*\circ \delta_x$ is an extreme point of $A^*_1$ for some $v^*\in {\rm ext}(E^*_1)$. Then  $ \ch(A)$ is a boundary for $A$, that is for each $f\in A$ there exists a point $x\in \ch(A)$ such that
$\|f\|_\infty=\|f(x)\|$.

By \cite[Lemma 7.2.2]{Flem2} for a locally compact Hausdorff space $X$ and a normed space $E$, if $S$ is a $T$-set in    $E$ and $x\in X$, then  the set
\[(S,x)=\{f\in C_0(X,E) : f(x)\in S \ ,\  \|f\|_\infty=\|f(x)\|\}\]
is a $T$-set in $C_0(X,E)$. Conversely, any $T$-set in $C_0(X,E)$ is of this form.

For any $T$-set $S$ in a normed space $E$, we put
\[\Gamma_S=\{v^*\in S(E^*) :  v^*(u)=\|u\|\;\; {\rm  for\;  all\;}  u\in S\}.\]  By \cite[Lemma 7.2.4]{Flem2} we have   $\Gamma_S\cap \ext(E_1^*)\neq \emptyset$.
 We should note that the Lemmas  7.2.2 and 7.2.4 in \cite{Flem2}  have  been stated  for the case that $E$ is a Banach space, however, the completeness of $E$ has no role in the proofs.
It is obvious that for any $T$-set $R$ in $E$, the corresponding set $\Gamma_R$ is a convex subset of $S(E^*)$ which is  norm additive.

The following proposition, which states some elementary properties of $T$-sets,  is easily verified. For the sake of completeness we state it here.
\begin{prop}\label{P0}
Let $E$ be a normed space over $\Bbb K$. Then

{\rm (i)} For any T-set $R$ in $E$ and $w^*\in \Gamma_R$ we have $R=\{u\in E : w^*(u)=\|u\|\}$.

{\rm (ii)} For distinct  T-sets $R_1$ and $R_2$ in $E$, if $R_1\cap R_2\neq \{0\}$, then for any $w_1^*\in \Gamma_{R_1}$ and any $w_2^*\in \Gamma_{R_2}$  we have $w_1^*\in \St(w_2^*)$.

{\rm (iii)} For distinct T-sets  $R_1$ and $R_2$ in $E$ we have  $\Gamma_{R_1}\cap \Gamma_{R_2}=\emptyset$.

{\rm (iv)} If $E^*$ is  strictly convex, then for any T-set $R$ in $E$, $\Gamma_R$ is a singleton.
\end{prop}

\begin{proof}
(i) It is obvious that $R\subseteq \{u\in E : w^*(u)=\|u\|\}$. Now since  for each $u_1, ..., u_n\in E$ with $w^*(u_i)=\|u_i\|$, $i=1,...,n$, we have \[\|u_1+\cdots +u_n\|=\|u_1\|+\cdots +\|u_n\|,\] the maximality of $R$ implies that $R=\{u\in E : w^*(u)=\|u\|\}$.

(ii)  Let  $u\in R_1\cap R_2$ be nonzero. Since $(w_1^*+w_2^*)(u)=2\|u\|$ it follows that $\|w_1^*+w_2^*\|=2$ that is $w_1^*\in \St(w_2^*)$.

(iii) Assume on the contrary that  $\Gamma_{R_1}\cap \Gamma_{R_2}\neq \emptyset$ and let   $w^*$ be a point in this intersection. Then by (i) we have
\begin{align*}
R_1=\{ u\in E : w^*(u)=\|u\|\}=R_2,
\end{align*}
which is a contradiction.

(iv) Let $R$ be a $T$-set in $E$ and $w_1^*, w_2^*\in \Gamma_R$ be distinct. Being $E^*$ strictly convex we have $\|w_1^*+w_2^*\|<2$ while by (i) for any nonzero $u\in R$ we have $w_1^*(u)+
w_2^*(u)=2\|u\|$, a contradiction.
\end{proof}

Two $T$-sets $S$ and $R$ in a normed space $E$ are said to be {\em discrepant} if either $S\cap R=\{0\}$, or there exists a T-set $L$ in $E$  such that  $R\cap L=S\cap L=\{0\}$. A normed space $E$ is said to satisfy {\em property} $(D)$ if any two T-sets in $E$ are discrepant (see Definition 7.2.10 in \cite{Flem2}).
Since in a strictly convex space $E$ each  $T$-set is of the form $\{tu:t\ge 0\}$ for some nonzero $u\in E$, it follows that for any two distinct $T$-sets $S$ and $R$ in $E$ we have $S \cap R=\{0\}$, that is all strictly  convex spaces have  property $(D)$.
For some examples of non-strictly convex normed spaces with this property, see Examples 7.2.11 in \cite{Flem2}.

\section{Main results}
In this section,  introducing a  property, called $(D_w)$,  which is weaker than the property $(D)$, we characterize surjective (not necessarily linear) isometries $T:A \longrightarrow B$ between certain subspaces $A$ and $B$ of of $C_0(X,E)$ and $C_0(Y,F)$ where $F$ satisfies the property $(D_w)$.

\begin{defn}\label{D'}
{\rm We say that a  normed space $E$  satisfies {\em  property $(D_w)$} if there exists a $T$-set $R_{0}$ in $E$  which is discrepant to  any other $T$-set in $E$, that is for any $T$-set $R$ in $E$ distinct from $R_0$  we have either $R_0\cap R=\{0\}$ or there exists a $T$-set $L$ in $E$ such that  $R_0\cap L=R\cap L=\{0\}$.}
\end{defn}

%%%%%%%%%%%%%%%%%%%
{\bf Remark.} (i) We should note that if $E$ is a normed space with a $T$-set $R_0$ such that $R_0\cap R=\{0\}$ for all $T$-sets $R$ in $E$ distinct from $R_0$, then $E$ clearly  satisfies the  property $(D)$. In particular,  if $E$ is a normed space whose unit sphere $S(E)$ contains an element $e$ with $\St(e)=\{e\}$, then $E$ has property $(D)$. For an example of a non-strictly convex space $E$ such that $\St(e)=\{e\}$  for some $e\in S(E)$, see \cite{Jam}.

(ii) If  $E$ is a normed space and $w_0^*\in \cup_{R} \Gamma_R$, where the union is taken over all $T$-sets of $E$, such that $\St(w_0^*)=\{w_0^*\}$ then, using Proposition \ref{P0}, for the  $T$-set $R_0$ in $E$ containing $w_0^*$ we have  $R_0\cap R=\{0\}$ for all $T$-sets $R$ in $E$ distinct from $R_0$. Hence $E$ has property $(D)$.  In particular, if $E^*$ is strictly convex, then any pair of distinct $T$-sets of $E$ has trivial intersection and consequently $E$ has property $(D)$.

%%%%%%%%%%%%%%%%%%%

\vspace*{.25cm}

It is clear that property $(D_w)$ is weaker than the  property $(D)$.  We give an example which shows  that the converse statement does not necessarily  hold.
\begin{exam}\label{exm Dw}{\rm 
Let $E$ be a normed space whose closed unit ball  is the  subset $K$ of $\Bbb R^3$ as in Figure 1 with the origin in the center of $K$. Indeed, since $K$ is a symmetric compact convex subset of $\BR^3$ and origin is an interior point of $K$, it suffices to consider $E=\Bbb R^3$ with the  norm $\|\cdot\|$ defined by  $\|0\|=0$ and for each nonzero point $x\in \BR^n$, $\|x\|= \frac{1}{\max\{t\in \BR: tx \in K\}}$. Then $K$ is the closed unit ball of $E$ with respect to this norm. 

  The set $K$  consists of a cube and two pyramid on up and down.  Hence the unit sphere of $E$  has  twelve maximal convex subsets (four faces of cube, four faces of upper pyramid and four faces of bottom pyramid), that is $E$ has twelve $T$-sets. We note that $E$ does not satisfy the  property 
$(D)$. Indeed, letting $R$ and $S$ be the $T$-sets corresponding to  two adjacent faces in the cube, we have  $R\cap S=\{0\}$. Meanwhile, the other $T$-sets clearly have non-empty intersection with at least one of $R$ or $S$. Thus $E$ does not satisfy the property $(D)$. However, considering $R_0$ as the  $T$-set corresponding to one of the upper pyramid faces we see that $E$ satisfies the  property $(D_w)$.}
\begin{figure}
\centerline{ \includegraphics[scale=0.55]{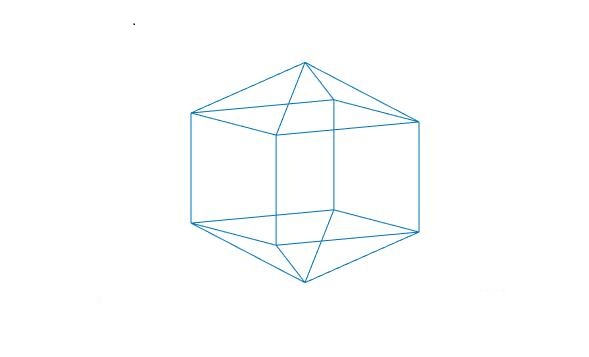}}
%\caption{Example \ref{exm Dw} }
\caption{}
\end{figure}
\end{exam}

For a   subspace $A$ of $C_0(X,E)$ we say that $A$ is {\em $E$-separating} if
for distinct points  $x_1\neq x_2$ in $X$  and any  $u \in E$ there
exists $f\in A$ such that $f(x_1)=u $ , $\|f\|_\infty=\|u\|$ and $f(x_2)=0$.  We say that $A$ is {\em completely  regular} if for any $x$ in $X$ , any $u \in E$ and any neighbourhood $U$ of $x$ there exists $f\in A$ such that $f(x)=u $, $\|f\|_\infty=\|u\|$ and $f=0$ on $X\backslash U$.
It is clear that any completely  regular subspace of $C_0(X,E)$ is $E$-separating.

For examples of completely regular subspaces of $C(K,E)$, where $K$ is a compact Hausdorff space,  we can refer to the spaces  ${\rm Lip}^\alpha (K,E)$ of   $E$-valued Lipschitz functions of order $\alpha\in [0,1]$ on a compact Hausdorff space $K$, the space $C^n([0,1],E)$ of $E$-valued $n$-times continuously differentiable functions on $[0,1]$, and its subspace ${\rm Lip}^n([0,1],E)$ consisting of those functions whose derivatives are also Lipschitz functions. The  space ${\rm AC}(K,E)$ of all absolutely continuous $E$-valued functions on the compact subset $K$ of the real-line is also completely regular.
On the other hand, by \cite{Cen}, for any locally compact Hausdorff space $X$, the kernel of each  continuous complex-valued  regular Borel measure on $X$ is a completely regular subspace of  $C_0(X)$.

Next Lemma gives the general form of $T$-sets in  $E$-separating  subspaces  of $C_0(X,E)$.

\begin{lem}\label{lem t-set A}
Let $X$ be a locally compact Hausdorff space, $E$ be a real or complex normed space  and $A$ be an $E$-separating  subspace  of $C_0(X,E)$. Then for any T-set $S$ in  $E$ and any point $x\in X$, the set $(S,x)\cap A$ is a T-set in $A$ and conversely, any T-set in $A$ is of this form.
\end{lem}

\begin{proof}
First assume that $R$ is a $T$-set in $A$. Since $R$ is  a norm additive subset of $C_0(X,E)$, it is contained in a $T$-set in $C_0(X,E)$. Hence  there exists a $T$-set $S$ in $E$  and a  point $x\in X$ such that $R\subseteq (S,x)$. We note that $R \subseteq (S,x)\cap A$ and clearly $(S,x)\cap A$ is a norm additive subset of $A$. Hence it follows from the maximality of $R$ that  $R=(S,x)\cap A$.

 To prove the converse statement, let $S$ be a $T$-set in $E$ and let $x\in X$. Since $(S,x)\cap A$ is a norm  additive subset of $A$ there exists  a T-set $R$ in  $A$ such that $(S,x) \cap A \subseteq R$. By the first part, there are a $T$-set $S_0$ in $E$ and a point $x_0\in X$ such that $R=(S_0,x_0)\cap A$. We claim  that $S=S_0$ and  $x=x_0$. Assume  that  $x\neq x_0$. Choosing a nonzero element  $e\in S$ it follows from the hypothesis that  there exists $f\in A$ such that $f(x)=e $, $\|f\|_\infty=\|e\|$ and $f(x_0)=0$. Thus  $f\in (S,x)\cap A\subseteq (S_0,x_0)\cap A$. Hence $\|f\|_\infty=\|f(x_0)\|=0$, a contradiction. This shows that $x_0=x$. Since $S$ and $S_0$ are both $T$-sets in $E$, to prove that $S=S_0$ it suffices to show that $S\subseteq S_0$.
 For this  suppose that  $e\in S$, and choose  $f\in A$ such that $f(x)=e$, $\|f\|_\infty=\|e\|$. Then  $f\in (S,x)\cap A\subseteq (S_0,x_0)\cap A$. In particular,  $e=f(x)=f(x_0)\in S_0$, that is  $S\subseteq S_0$, as desired.
\end{proof}

\begin{prop}\label{pro ch-boun}
Let $X$ be a locally compact Hausdorff space, $E$ be a real or complex normed space  and $A$ be an $E$-separating  subspace  of $C_0(X,E)$. Then for any $x\in X$ and any $T$-set $S$ in  $E$ there exists  $v^*\in \Gamma_S\cap \ext(E_1^*)$ such that $v^*\circ \delta_x\in \ext(A_1^*)$. In particular,  $\ch(A)=X$.
\end{prop}

\begin{proof}
Let $x\in X$ and $S$ be a T-set in $E$. By Lemma \ref{lem t-set A}, $(S,x)\cap A$ is a T-set in $A$. Hence  that there exists $l\in \ext(A_1^*)$ such that
$l(f)=\|f\|_\infty$ for all $f\in (S,x)\cap A$.
Since $l\in \ext(A_1^*)$,  there are $z\in X$ and $v^*\in \ext(E_1^*)$ such that $l=v^*\circ\delta_z$. Thus for all $f\in (S,x)\cap A$  we have
\begin{equation}\label{eq 3.1}
v^*(f(z))=v^*\circ\delta_z(f)=l(f)=\|f\|_\infty.
\end{equation}
We show that $z=x$ and $v^*\in\Gamma_S$. Suppose that $z\neq x$ and  let $u$ be a nonzero element of $S$. By hypothesis, there exists $f\in A$ such that
$f(x)=u$, $ f(z)=0$ and $\|f\|_\infty=\|u\|$. Then $f\in (S,x)\cap A$ and it follows from (\ref{eq 3.1}) that $0=v^*(f(z))=\|f\|_\infty=\|u\|$, a contradiction. Hence $z=x$.
To  show that $v^*\in\Gamma_S$, let $u\in S$ and choose  $f\in A$ such that $f(x)=u$ and $\|f\|_\infty=\|u\|$. Then  $f\in (S,x)\cap A$ and it follows from (\ref{eq 3.1}) that  $v^*(u)=\|u\|$. This shows that $v^*\in \Gamma_S$, as desired.
\end{proof}

\begin{prop}\label{pro isom}
Let $X$ and $Y$ be locally compact Hausdorff spaces, $E$ and $F$ be real or complex normed spaces (not necessarily complete) and $A$ and $B$ be $E$-separating and $F$-separating  subspaces of $C_0(X,E)$ and $C_0(Y,F)$, respectively. Let $T : A\longrightarrow B$ be a surjective real-linear isometry. If $x\in X$ and $y\in Y$ such that  $T((S,x)\cap A)=(R,y)\cap B$ where $S$ and $R$ are T-sets in $E$ and $F$, respectively, then there are $v^*\in \Gamma_S\cap {\rm ext}(E_1^*)$ and $w^*\in \Gamma_R\cap {\rm ext}(F_1^*)$ such that $T^*(w^*\circ\delta_y)=v^*\circ\delta_x$.
\end{prop}
\begin{proof}
Since the real dual  of a complex normed space is isometrically isomorphic to its complex dual, without loss of generality we assume that $E$ and $F$ are real normed spaces. Let $A^*$ and $B^*$ denote the (real) duals of $A$ and $B$, respectively, and  $T^*: B^*\longrightarrow A*$ be the adjoint of $T$ as a bounded real-linear operator. Then $T^*$ is a surjective real-linear isometry. By Proposition \ref{pro ch-boun} for the T-set $R$ in $F$ and the point $y\in Y$ there exists $w^*\in \Gamma_R$ such that $w^*\circ\delta_y\in \ext(B_1^*)$. Since $T^*(w^*\circ\delta_y)$ is an extreme point of the unit ball of $A^*$ there are $v^*\in\ext(E_1^*)$ and $x_0\in X$ such that $T^*(w^*\circ\delta_y)=v^*\circ\delta_{x_0}$, that is
\begin{equation}\label{eq 3.6}
w^*(Tf(y))=v^*(f(x_0))\quad\quad (f\in A).
\end{equation}
It suffices to show that $v^*\in \Gamma_S$ and $x_0=x$. Suppose that $x_0\neq x$. Considering a nonzero element $u\in S$ we can find a function $f\in A$ such that
$f(x)=u$, $f(x_0)=0$ and $\|f\|_\infty=\|u\|$. Hence  $f\in (S,x)\cap A$ and consequently $Tf\in (R,y)\cap B$, that is $\|Tf(y)\|=\|Tf\|_\infty$ and $Tf(y)\in R$. Thus, using  (\ref{eq 3.6}) we  have
\begin{align*}
\|u\|=\|f\|_\infty=\|Tf\|_\infty=\|Tf(y)\|=w^*(Tf(y))=v^*(f(x_0))=0,
\end{align*}
which is a contradiction. This implies that  $x_0=x$.

 Now to show that $v^*\in\Gamma_S$, let  $u\in S$ be given and choose  $f\in A$ such that $f(x)=u$ and $\|f\|_\infty=\|u\|$. Then  $f\in (S,x)\cap A$ and consequently $Tf\in (R,y)\cap B$. Using  (\ref{eq 3.6}), once again, it  follows that  $v^*(u)=\|u\|$, as desired.
\end{proof}

\begin{lem}\label{lem intersect}
Let $X$ be a locally compact Hausdorff space, $E$ be a real or complex normed space and  $A$ be a completely regular subspace of $C_0(X,E)$.  If $S$ and $R$ are $T$-sets in $E$ and $x,z\in X$ such that $(S,x) \cap (R,z) \cap A=\{0\}$,  then $x=z$ and $S\cap R=\{0\}$.
\end{lem}
\begin{proof}
First assume that  $x\neq z$ and choose  disjoint  neighbourhoods $U$  and $V$ of $x$ and  $z$, respectively in $X$. Since any T-set is  closed under positive multiples, there are nonzero elements  $u\in S$ and $v\in R$ such that $\|u\|=\|v\|$. By hypothesis,  there exist $f,g\in A$ such that
$f(x)=u$, $\|f\|_\infty=\|f(x)\|$ and $ f=0$ on $X\backslash U$, and similarly $g(z)=v$,  $\|g\|_\infty=\|g(z)\|$ and $ g=0$ on $X\backslash V$. We put $F=f+g$. Then clearly  $\|F\|_\infty=\|u\|=\|v\|$, $F(x)=u$ and $F(z)=v$. Hence  $F\in (S,x) \cap (R,z) \cap A=\{0\}$,  a contradiction.

To show that $S\cap R=\{0\}$, let  $u$ be a nonzero element in $S\cap R$. By hypothesis, there exists $f\in A$ such that $f(x)=u$ , $\|f\|_\infty=\|u\|$, that is  $f\in (S,x) \cap (R,z) \cap A=\{0\}$, which is again a contradiction.
\end{proof}

Let $X$ and $Y$ be locally compact Hausdorff spaces, $E$ and $F$ be real or complex normed spaces (not necessarily complete), $A$ and $B$ be $E$-separating and $F$-separating  subspaces of $C_0(X,E)$ and $C_0(Y,F)$, respectively. Let $T : A\longrightarrow B$ be a surjective real-linear isometry. Being $T^{-1}$ an isometry, it  maps any $T$-set in $B$ to a  $T$-set in $A$. Hence   for any $T$-set $R$ in $F$ and a point $y\in Y$,  there exist a $T$-set $S$ in $E$ and a point  $x\in X$  such that $T^{-1}((R,y)\cap B)=(S,x)\cap A$. Using the separating property,  it is easy to see that the $T$-set $S$ and  the point $x\in X$ satisfying this equality are uniquely determined. Hence for each $T$-set $R$ in $F$ we can define a  function  $\varphi_R : Y\longrightarrow X$ such that for each $y\in Y$, there exists a $T$-set $S$ in $E$ satisfying    $T^{-1}((R,y)\cap B)=(S,\varphi_R(y))\cap A$.

\begin{lem}\label{well defn}
Let $X$ and $Y$ be locally compact Hausdorff spaces, $E$ and $F$ be real or complex normed spaces not assumed to be  complete. Let $A$ be a completely  regular subspace of $C_0(X,E)$, $B$ be an $F$-separating subspace of $C_0(Y,F)$ and $T : A\longrightarrow B$ be a surjective real-linear isometry. If $F$ satisfies the  property $(D_w)$,  then $\varphi_{R_1}=\varphi_{R_2}$ for all $T$-sets $R_1$ and $R_2$ in $F$.
\end{lem}

\begin{proof}
Since $F$ satisfies the property $(D_w)$, there exists a $T$-set $R_0$ which is discrepant to all $T$-sets in $F$. It suffices to show that $\varphi_{R}=\varphi_{R_0}$ for all $T$-sets $R$ in $F$. Let $R$ be a given $T$-set in $F$ and $y\in Y$. Put $x=\varphi_R(y)$ and $x_0=\varphi_{R_0}(y)$. Hence there are $T$-sets $S$ and $S_0$ in $E$ such that   $T^{-1}((R,y)\cap B)=(S,x)\cap A$ and $T^{-1}((R_0,y)\cap B)=(S_0,x_0)\cap A$.  Since the T-sets $R$ and $R_0$ are discrepant, the same arguments  as in \cite[Theorem 7.2.13]{Flem2} together with Lemma \ref{lem intersect} imply that $x=x_0$.
\end{proof}

Using the above lemma, in the case that  $A$ is completely  regular, $B$ is  $F$-separating and $F$ satisfies the property $(D_w)$, we can define a function $\varphi: Y \longrightarrow X$ such that for each  $T$-set $R$ in $F$ there exists a $T$-set $S$ in $E$ satisfying $T^{-1}((R,y)\cap B)=(S,\varphi(y))\cap A$.   It is easy to see that  $\varphi$ is surjective. Indeed, given $x\in X$, let $S$ be an arbitrary $T$-set in $E$. Then there exist a $T$-set $R$ in $F$ and  a point $y\in Y$ such that $T((S,x)\cap A)=(R,y)\cap B$, that is $T^{-1}(R,y)\cap B)=(S,x)\cap A$ which concludes that $\varphi(y)=\varphi_R(y)=x$.

\begin{lem}\label{varphi}
Under the assumptions of Lemma \ref{well defn} if $f\in A$ and $y\in Y$ such that $f(\varphi(y))=0$,  then $(Tf)(y)=0$.
\end{lem}
\begin{proof}
We put $x=\varphi(y)$ and $u=Tf(y)$. Assume on the contrary that  $u\neq 0$. Choose a $T$-set  $R$ in $F$ containing  $u$. Then, by the definition of $\varphi$, there exists a T-set $S$ in $E$ such that $T^{-1}((R,y)\cap B)=(S,x) \cap A$. Now by Proposition \ref{pro isom} there are $v^*\in \Gamma_S$ and $w^*\in \Gamma_R$ such that $T^*(w^*\circ\delta_y)=v^*\circ\delta_x$. Hence we  have
\[\|u\|=w^*(u)=w^*(Tf(y))=v^*(f(x))=0,\]   a contradiction.
\end{proof}

\begin{theorem}\label{thm main}
Let $X$ and $Y$ be locally compact Hausdorff spaces, $E$ and $F$ be real or complex normed spaces (not necessarily complete), $A$ be  a completely regular subspace of $C_0(X,E)$ and $B$ be an $F$-separating subspace of $C_0(Y,F)$. Let $T : A\longrightarrow B$ be a surjective real-linear isometry.  If $F$ satisfies the property $(D_w)$, then there exist a  continuous map  $\varphi:  Y
\longrightarrow X$, a family $\{V_y\}_{y\in Y}$ of bounded real-linear operators from $E$ to $F$ with $\|V_y\|\le 1$ such that for each $y\in Y$
\[Tf(y)= V_y(f(\varphi(y))) \qquad (f\in A).\]
Moreover, if  $E$ also satisfies the property $(D_w)$ and $B$ is completely regular, then $\varphi$ is a homeomorphism and each $V_y$ is a surjective isometry.
\end{theorem}
\begin{proof}
Let  $\varphi:Y \longrightarrow X$ be the function defined  before. For each $y\in Y$ and each $e\in E$ by the hypotheses there exists $f\in A$ such that $f(\varphi(y))=e$.  Put  $V_y(e)=Tf(y)$. We note that, by Lemma \ref{varphi}, the definition of $V_y(e)$ is independent of  the function $f\in A$ satisfying $f(\varphi(y))=e$.  Then $V_y: E \longrightarrow F$ is clearly  a real-linear operator satisfying
 \begin{align} \label{eq 3.4}
 Tf(y) = V_y(f(\varphi(y)) \;\;\;\;\; (f\in A).
 \end{align}
Since for each $y\in Y$ and $e\in S(E)$ we can choose a function $f\in A$ with $\|f\|_\infty=1$ and $f(\varphi(y))=e$, it follows easily that $\|V_y\|\le 1$.

To prove that $\varphi$ is continuous, assume that $y_0\in Y$ and $U$ is an open neighbourhood of $\varphi(y_0)$ in $X$.
 Since $V_{y_0}\neq 0$ there exists $e\in E$ such that $V_{y_0}(e)\neq 0$. Choose  $f\in A$  such that $f(\varphi(y_0))=e$ and $f=0$ on $X\backslash U$. Then $W=\{y\in Y: Tf(y) \neq 0\}$ is a neighbourhood of $y_0$ and the equality (\ref{eq 3.4}) implies that $\varphi(W) \subseteq U$. Hence $\varphi$ is continuous.

 The second part of the theorem is easily verified.
\end{proof}

In the next theorem, we consider the compact case and, using Theorem \ref{thm main}, we  characterize surjective isometries between certain subspaces $A$ and $B$ of $C(X,E)$ and $C(Y,F)$, respectively endowed with some norms rather than supremum norms. Motivated by  the property {\bf P}  introduced in \cite{Ar} and the property {\bf Q} introduced in \cite{Bot1} for an isometry $T:A \longrightarrow B$, we consider the  property (St) introduced in the earlier work of the authors \cite{MS}. We have compared the new defined property (St) with the properties  {\bf P} and  {\bf Q} in \cite{MS}. Indeed, the property {\bf Q} implies the property (St) and in the case where $F$ is strictly convex (this is assumed in \cite{Ar}) the property {\bf P} also implies the property (St).

\begin{defn}\cite[Definition 3.4]{MS}
Let $X$ and $Y$ be compact Hausdorff spaces and let $E$ and $F$ be real or complex normed spaces. Assume that  $A$ and $B$ are  subspaces of $C(X,E)$ and $C(Y,F)$, respectively, equipped with the norms of the form
\[ \|\cdot \|_A= \max(\|\cdot\|_\infty, p(\cdot)) \;{\rm and}\; \|\cdot \|_B= \max(\|\cdot\|_\infty, q(\cdot)),\]
where $p$ and $q$ are seminorms on $A$ and $B$, respectively, whose kernels contain the constant functions. We say that a surjective real-linear isometry  $T:A \longrightarrow B$ has property  {\rm (St)} if
\begin{quote}
{\rm (St)} For each $u\in F$ and $y_0\in Y$ there exists $v\in
S(E)$ such that $\|T\hat{v}(y_0)+u\|>\|u\|$, i.e.
$\frac{T(\hat{v})(y_0)}{\|T(\hat{v})(y_0)\|} \in \St_w(u)$.
\end{quote}
\end{defn}

 \begin{prop}\label{Pr}\cite[Proposition 3.5]{MS}
Let $X$ and $Y$ be compact Hausdorff spaces and let $E$ and $F$ be real or complex  normed spaces, not assumed to be complete. Assume that
 $A$ and $B$ are subspaces of $C(X,E)$ and $C(Y,F)$, respectively containing constants and  $\|\cdot\|_A$ and $\|\cdot \|_B$ are norms on $A$ and $B$ such that
 \[ \|\cdot \|_A= \max(\|\cdot\|_\infty, p(\cdot)) \;{\rm and}\; \|\cdot \|_B= \max(\|\cdot\|_\infty, q(\cdot))\]
for some seminorms $p$ and $q$ on $A$ and $B$, respectively, whose kernels contain the constants.
If  $T:A \longrightarrow B$ is a surjective real-linear isometry and $T$ and $T^{-1}$ satisfy {\rm (St)}, then  $T$ is an isometry with respect to the supremum norms on $A$ and $B$.
  \end{prop}

Using the above proposition we get the next result concerning surjective isometries between completely regular subspaces of functions with respect to some  norms.

\begin{theorem}\label{thm app}
Let $X$ and $Y$ be compact Hausdorff spaces, $E$ and $F$ be real
or complex normed spaces, not necessarily complete such that $F$ satisfies the property $(D_w)$. Let $A$  and $B$ be subspaces of  $C(X,E)$ and  $C(Y,F)$, respectively  containing
the constants such that $A$ is completely regular and $B$ is $F$-separating. Let $\|\cdot\|_A$ and $\|\cdot\|_B$ be norms on $A$ and $B$, respectively, of the  form
\[ \|\cdot \|_A= \max(\|\cdot\|_\infty, p(\cdot)) \;{\rm and}\; \|\cdot \|_B= \max(\|\cdot\|_\infty, q(\cdot))\]
for some seminorms $p$ and $q$ on $A$ and $B$, respectively, whose kernels contain the constants.
Then for any surjective real-linear isometry $T: A\longrightarrow B$  such that $T$ and $T^{-1}$ satisfy {\rm (St)} there exist a  surjective continuous map  $\varphi:  Y
\longrightarrow X$, a family $\{V_y\}_{y\in Y}$ of bounded real-linear operators from $E$ to $F$ with $\|V_y\|\le 1$ such that for each $y\in Y$
\[Tf(y)= V_y(f(\varphi(y))) \qquad (f\in A).\]
Moreover, if  $E$ also satisfies the property $(D_w)$ and $B$ is completely regular,  then $\varphi$ is a homeomorphism and each $V_y$ is a surjective isometry.
\end{theorem}

\begin{proof}
It  follows immediately from Proposition \ref{Pr} and Theorem \ref{thm main}.
\end{proof}

\end{document}